\documentclass[12pt]{article}
\usepackage{amsthm}
\usepackage{amsmath,bm}
\usepackage{enumerate}
\usepackage{amssymb}
\usepackage{latexsym}
\usepackage{amsfonts}
\usepackage{color}
\usepackage{secdot}
\usepackage{mathrsfs}
\usepackage{subfigure}
\usepackage{epsfig}
\usepackage{natbib}
\usepackage{rotating} 
\newtheorem{theorem}{Theorem}[section]
\newtheorem{setup}{Set-up}[section]
\newtheorem{counterexample}{Counterexample}[section]

\newtheorem{corollary}{Corollary}[section]
\newtheorem{definition}{Definition}[section]
\newtheorem{lemma}{Lemma}[section]
\newtheorem{remark}{Remark}[section]
\newtheorem{proposition}{Proposition}[section]

\usepackage{hyperref}
\usepackage{multirow}
\usepackage{multicol}
\usepackage{booktabs,tabularx}
\newcolumntype{x}{>{\raggedright\arraybackslash}X}
\makeatletter \oddsidemargin  0.95in \evensidemargin 0.95in
\begin{document}
\title{Ordering results between two finite arithmetic mixture models with multiple-outlier location-scale distributed components}
\author{{\large{Raju {\bf Bhakta}$^{1}$\thanks {Email address: bhakta.r93@gmail.com,~raju\_bhakta.maths@yahoo.com},~Nuria {\bf Torrado}$^{2}$\thanks {Email address (corresponding author): nuria.torrado@uam.es, nuria.torrado@icmat.es},~Sangita {\bf Das}$^{3}$\thanks {Email address: sangitadas118@gmail.com},~and~Suchandan {\bf {Kayal}$^{4}$}\thanks {Email address: kayals@nitrkl.ac.in,~suchandan.kayal@gmail.com}}}}
\maketitle
\noindent{\it$^{1}$Department of Mathematics, Indian Institute of Technology Roorkee, Roorkee-247667, Uttarakhand, India.}\\
{\it$^{2}$Department of Mathematics at Universidad Aut\'{o}noma de Madrid and Institute of Mathematical Sciences, 28049 Madrid, Spain.}\\
{\it$^{3}$Theoretical Statistics and Mathematics Unit, Indian Statistical Institute, Bangalore-560059, India.}\\
{\it$^{4}$Department of Mathematics, National Institute of Technology Rourkela, Rourkela-769008, Odisha, India.}

\begin{center}
\noindent{\bf Abstract}
\end{center}
In this article, we introduce finite mixture models (FMMs) renowned for capturing population heterogeneity. 
Our focus lies in establishing stochastic comparisons between two arithmetic (finite) mixture models, employing 
the vector majorization concept in the context of various univariate orders of magnitude, transform, and variability.
These comparisons are conducted within the framework of multiple-outlier location-scale models. Specifically, 
we derive sufficient conditions for comparing two finite arithmetic mixture models with components distributed in a
multiple-outlier location-scale model.  \\

\noindent{\bf Keywords:}
Finite mixture model (FMM); univariate stochastic orders; majorization; location-scale model
\\
\\
{\bf Mathematics Subject Classification:} 60E15, 90B25.

\section{Introduction}
The study of homogeneous populations with single-component distributions has been a focal point for researchers for a long period. 
However, real-world populations are rarely homogeneous; rather, they encompass data from diverse and heterogeneous subpopulations. 
The first attempt to model such heterogeneous data using finite mixture models (FMMs) was undertaken by \cite{newcomb1886generalized} and 
\cite{pearson1894contributions}. Specifically, \cite{newcomb1886generalized} employed normal mixtures to model outliers, while 
\cite{pearson1894contributions} proposed parameter estimates for a finite mixture of two normal distributions.
Recently, a new family of bivariate distributions has been proposed by \cite{Slovaca23}, comprising a mixture of independent Gamma random variables and the bivariate generalized Lindley distribution.
In various fields of science and technology, FMMs have proven to be useful in current applications. A recent study by \cite{AMM23-Pal} employed a mixture model to analyze data involving dependent competing risks, with a specific focus on its application in the treatment of diabetic retinopathy. Another application is for components produced in different batches, which inevitably exhibit heterogeneity due to environmental variations, material differences, machinery factors, and human errors. As a result, the mass-produced components constitute a heterogeneous population consisting of distinct subpopulations. Therefore, components within a randomly selected batch can be attributed to these subpopulations with certain probabilities. In this context, lifetime data originating from a finite number of different subpopulations can be effectively modeled using FMMs.
For more comprehensive insights, refer to \cite{finkelstein2008failure} 
and \cite{cha2013failure}.

Various researchers have previously delved into the examination of stochastic comparisons for mixture models. For example, \cite{navarro2016stochastic} established sufficient conditions for comparing the hazard rate order and the likelihood ratio order of generalized mixtures. 
Meanwhile, \cite{amini2017stochastic} employed various orders, including hazard rate, reversed hazard rate, likelihood ratio, mean residual lifetime, and mean inactivity time, to investigate stochastic comparisons between two classical FMMs. These analyses were conducted under the condition that the baseline distributions are arranged in accordance with certain stochastic ordering, such as $F_1\geq_{hr}\cdots\geq_{hr}F_n$.
Under similar conditions of stochastic ordering among the baseline distributions, \cite{navarro2017stochastic} identified both necessary and sufficient conditions for stochastic comparisons of generalized distorted distributions. Additionally, they illustrated the applicability of these findings in comparing FMMs with component-wise order constraints, such as $F_1\geq_{st}\cdots\geq_{st}F_n$.
 \cite{hazra2018stochastic} employed multivariate chain majorization to establish stochastic comparison results for two FMMs with random variables following the proportional hazard rate (PHR), proportional reversed hazard rate (PRH), or accelerated lifetime (AL) models.
\cite{asadi2019alpha} delved into the failure rate properties of a model termed the alpha-mixture model ($\alpha$-MM), which combines elements of the arithmetic mixture model and the mixture of failure rate model.
In a related vein, \cite{barmalzan2021stochastic} investigated stochastic comparisons between two $\alpha$-FMMs with respect to various stochastic orders. 
\cite{nadeb2020new} explored comparisons for finite mixture models (FMMs) within the framework of the usual stochastic order, considering subpopulations that follow the PHR,  PRH model, or the scale models.
\cite{shojaee2022stochastic} investigated two $n$-component generalized $\alpha$-finite mixture models ($\alpha$-FMMs) with a shared mixing proportion vector, considering two scenarios: one where $\alpha_i\leq 0$ and another where $0<\alpha_i<1$.
\cite{barmalzan2022orderings} conducted a study that explores various stochastic orders of magnitude among FMMs for location-scale ($\mathcal{LS}$) families of distributed components. This investigation is carried out under chain majorization conditions applied to the matrix of parameters.
\cite{kayal2023some} obtained some ordering results between two FMMs considering general parametric families of distributions. Mainly, the authors established sufficient conditions for usual stochastic order based on $p$-larger order and reciprocally majorization order.
Moreover, \cite{bhakta2024stochastic} considered similar general parametric families of distributions as in \cite{kayal2023some}, and then examined various ordering results with respect to usual stochastic order, hazard rate order, and reversed hazard rate order between two FMMs. 
Recently, \cite{panja2024comparisons} derived stochastic results for coherent systems by incorporating component lifetimes modeled under the proportional odds (PO) framework, allowing for any lifetime distribution to act as the baseline distribution.


In this paper, inspired by the research of \cite{barmalzan2022orderings}, we conduct a comparative analysis of two FMMs, considering various univariate orders of magnitude, transform, and variability. The scenario involves mixing components following multiple-outlier $\mathcal{LS}$ families of distributions.
In particular, 
a random variable $X$ is considered to follow a $\mathcal{LS}$ family of distributions if its cumulative distribution function (cdf) is expressed as
\begin{eqnarray}\label{eq_ls}
F_X(x)\equiv F(x;\sigma,\lambda)=F\left(\frac{x-\sigma}{\lambda}\right),~~x>\sigma,~\lambda>0, 
\end{eqnarray}
where $\sigma\in \mathbb{R}$ is referred to as the location parameter, and $\lambda$ is termed the scale parameter.
Given our consideration of nonnegative independent random variables in this study, the support of the baseline cdf $F$ spans $[0,\infty)$. Consequently, the location parameter $\sigma$ must be greater than or equal to $0$.
It is worth mentioning that introducing a location parameter greater than zero can be interpreted in several ways.  
For instance, in reliability and life testing studies, the failure of a unit does not typically startat time $t=0$; instead, 
it often begins at a specific time, such as $t=\sigma>0$. In supply chain management studies, a non-zero threshold value is employed 
to determine when the lead time for a product starts. Additionally, in the context of insurance, the initiation of a claim under a 
health insurance policy does not happen immediately after policy activation. Instead, the claim is often initiated after a specific duration, like one year.

The article is structured as follows:
In Section 2, we introduce fundamental concepts, definitions, and crucial lemmas that will serve as the foundation for the subsequent sections.
Section 3 presents essential properties of finite mixture models (FMMs) adhering to multiple-outlier $\mathcal{LS}$ families of distributions.
Section 4 outlines the main results concerning stochastic comparisons, focusing on various univariate orders of magnitude, transform, and variability.
Lastly, in Section 5, we provide some concluding remarks.

Throughout the paper, we refer to a function $g:\mathbb{R}\rightarrow\mathbb{R}$ as increasing (decreasing) if and only if $g(x)\leq~(\geq)~g(y)$ holds true for all $x\leq y$. Here, the term ``increasing" and ``decreasing" refer to as ``nondecreasing" and ``nonincreasing", respectively. Also, ``$\stackrel{sign}{=}$" is used to indicate that the signs on both sides of an equality are the same. To represent the partial derivative of 
a function $\mathcal{G}:\mathbb{R}_n^+\rightarrow\mathbb{R}$
 with respect to its $k$th component, we use the notation ``$\mathcal{G}_{(k)}(\boldsymbol{x})=\partial \mathcal{G}(\boldsymbol{x})/\partial x_k$", where $k= 1,\ldots,n$. For any differentiable function $s:\mathbb{R}\rightarrow\mathbb{R}$, we write $s^{\prime}(t)$ to represent the first order derivative of $s(t)$ with respect to $t$. Denote $\mathbb{R}=(-\infty,+\infty)$, $\mathbb{R}^+=[0,+\infty)$, $\mathbb{R}_n=(-\infty,+\infty)^n$, and $\mathbb{R}_n^+=[0,+\infty)^n$.

\section{Definitions and some prerequisites}
In this section, we provide some fundamental definitions 
and useful lemmas that play a crucial role in the subsequent sections. Consider two continuous nonnegative random variables, $X$ and $Y$, each characterized by probability density functions $f_X$ and $f_Y$, cumulative distribution functions $F_X$ and $F_Y$, quantile functions $F^{-1}_X$ and $F^{-1}_Y$, survival functions $\bar{F}_X$ and $\bar{F}_Y$, hazard rate functions $h_X$ and $h_Y$, as well as reversed hazard rate functions $\tilde{h}_X$ and $\tilde{h}_Y$. The following definitions will be instrumental in our derivations.
Firstly, we revisit the definitions of certain stochastic orders of magnitude that will be examined in Section 4 for FMMs following multiple-outlier $\mathcal{LS}$ families of distributions.

\begin{definition}
A random variable $X$ is said to be smaller than $Y$ in the sense of
\begin{itemize}
	\item[i)] usual stochastic order (denoted as $X\leq_{st}Y$) if $\bar{F}_X(x)\leq\bar{F}_Y(x)$ for all $x\in\mathbb{R}^+$;
	
	\item[ii)] hazard rate order (denoted as $X\leq_{hr}Y$) if $\bar{F}_Y(x)/\bar{F}_X(x)$ is increasing in $x$, for all $x\in\mathbb{R}^+$ or equivalently $h_X(x)\geq h_Y(x)$, for all $x\in\mathbb{R}^+$;   
	
	\item[iii)] reversed hazard rate order (denoted as $X\leq_{rh}Y$) if $F_Y(x)/F_X(x)$ is increasing in $x$, for all $x\in\mathbb{R}^+$ or equivalently $\tilde{h}_X(x)\leq\tilde{h}_Y(x)$, for all $x\in\mathbb{R}^+$;    
	
	\item[iv)] likelihood ratio order (denoted as $X\leq_{lr}Y$) if $f_Y(x)/f_X(x)$ is increasing in $x$, for all $x\in\mathbb{R}^+$.    
\end{itemize}
\end{definition}  

It is important to note that the following chains of implications hold:
 $X\leq_{lr}Y\Rightarrow X\leq_{hr}Y \Rightarrow X\leq_{st}Y$ and $X\leq_{lr}Y\Rightarrow X\leq_{rh}Y \Rightarrow X\leq_{st}Y$.  
Secondly, we recall some transform and variability orders that are used in this article.

\begin{definition}
	A random variable $X$ is said to be smaller than $Y$ in the sense of
	\begin{itemize}	
		
		\item[i)] star order (denoted as $X\leq_{*}Y$) if $F^{-1}_Y(F_X(x))/x$ is increasing in $x$, for all $x\in\mathbb{R}^+$;  
		
		\item[ii)] Lorenz order (denoted as $X\leq_{Lorenz}Y$) if $\frac{1}{E(X)}\int_{t}^{1}F^{-1}_X(x)dx\leq\frac{1}{E(Y)}\int_{t}^{1}F^{-1}_Y(x)dx$, for all $t\in[0,1]$ and for which the expectations are exist;
		
		\item[iii)] dispersive order (denoted as $X\leq_{disp}Y$) if $F^{-1}(\beta)-F^{-1}(\alpha)\leq F_{Y}^{-1}(\beta)-F_{Y}^{-1}(\alpha)$, for all
			$0<\alpha\leq\beta<1$ and $\forall~x\in\mathbb{R}^+$;
		
		\item[iv)] right spread order (denoted as $X\leq_{RS}Y$) if $\int_{F^{-1}(p)}^{\infty}\bar{F}(x)dx\leq\int_{F_{Y}^{-1}(p)}^{\infty}\bar{F_{Y}}(x)dx$ for all $0\leq p\leq 1$. 
%
	\end{itemize}
\end{definition}  

It is well known that $X\leq_{*}Y\Rightarrow X\leq_{Lorenz}Y$. 
For a comprehensive exploration of various stochastic orderings, refer to \cite{shaked2007stochastic}. 
Next, we recall some well known proportional
ageing notions. 

\begin{definition}
	A random variable $X$, or analogously, its distribution function $ F$ is   
	
	\begin{itemize}
		\item[i)] increasing (decreasing) proportional failure rate, denoted by IPFR
		(DPFR),  if and only if, $xh(x)$ is  increasing (decreasing) for all $x$.
		
		\item[i)] increasing (decreasing) proportional reversed failure  rate,
		denoted by IPRFR (DPRFR),  if and only if, $x\tilde{h}(x)$ is increasing (decreasing)
		for all $x$.
		
		\item[ii)] increasing (decreasing) proportional likelihood ratio,  denoted by
		IPLR (DPLR),  if and only if, $-xf^{\prime }(x)/f(x)$ is increasing
		(decreasing) for all $x$. 
	\end{itemize}
\end{definition}

The monotonicity of these functions
and their relationship with other notions of ageing have been studied in \cite{oliveira2015proportional}.
In particular, in their Theorem 5.1(a), they proved that if $X$ is a random variable with
a decreasing PFR function, then $X$ also has a decreasing PRFR function.
Moreover, they showed that if $X$ is DPLR, then both the PFR and
PRFR functions are also decreasing.

On the other hand, 
the theory of majorization orders find diverse applications across different domains of probability and statistics.
Next, we delve into the concept of majorization along with the associated order structures.

\begin{definition}\label{definition2.2}
	Consider $\boldsymbol{u}=(u_1,\ldots,u_n)$ and $\boldsymbol{v}=(v_1,\ldots,v_n)$ be two real vectors coming from $\mathbb{R}_n$. Further, assume that $u_{(1)}\leq\cdots\leq u_{(n)}$ and $v_{(1)}\leq\cdots\leq v_{(n)}$ denote the respective increasing arrangements of the components of $\boldsymbol{u}$ and $\boldsymbol{v}$, respectively. 
The vector $\boldsymbol{u}$ is said to be 
\begin{itemize}
	\item[i)] majorized by the vector $\boldsymbol{v}$ (denoted as $\boldsymbol{u}\stackrel{m}{\preccurlyeq}\boldsymbol{v}$) if $\sum_{i=1}^{j}u_{(i)}\geq\sum_{i=1}^{j}v_{(i)}$, for all $j=1,\ldots,n-1$, and $\sum_{i=1}^{n}u_{(i)}=\sum_{i=1}^{n}v_{(i)}$;   
	
	\item[ii)] weakly supermajorized by the vector $\boldsymbol{v}$ (denoted as $\boldsymbol{u}\stackrel{w}{\preccurlyeq}\boldsymbol{v}$) if $\sum_{i=1}^{j}u_{(i)}\geq\sum_{i=1}^{j}v_{(i)}$, for all $j=1,\ldots,n$;    
	
	\item[iii)] weakly submajorized by the vector $\boldsymbol{v}$ (denoted as $\boldsymbol{u}\preccurlyeq_{w}\boldsymbol{v}$) if $\sum_{i=j}^{n}u_{(i)}\leq\sum_{i=j}^{n}v_{(i)}$, for all $j=1,\ldots,n$.   
\end{itemize}
\end{definition}           

It can be verified, based on definition \ref{definition2.2}, that the majorization order implies both weak (super and sub) majorization orders.
Now, we introduce the concept of Schur functions in connection with the majorization order.
\begin{definition}
	A real valued function $\varphi$ defined on a set $\mathcal{A}\subseteq\mathbb{R}_n$ is said to be Schur-convex (Schur-concave) on $\mathcal{A}$ if and only if 
	\begin{eqnarray*}
		\boldsymbol{u}\stackrel{m}{\preccurlyeq}\boldsymbol{v}\Rightarrow\varphi(\boldsymbol{u})\leq(\geq)~\varphi(\boldsymbol{v})~\mbox{for all}~\boldsymbol{u}, \boldsymbol{v}\in\mathcal{A}. 
	\end{eqnarray*}   
\end{definition} 
For more details on majorization, its related orders and their applications, one may refer to \cite{marshall2011inequalities}.
The following notations will be consistently employed throughout the remainder of this paper.
\begin{itemize}
	\item[] $\mathcal{D}_n^+=\{\boldsymbol{u}=(u_1,\ldots,u_n)\in\mathbb{R}_n:u_1\geq\ldots\geq u_n>0\}$; 
	
	\item[] $\mathcal{E}_n^+=\{\boldsymbol{u}=(u_1,\ldots,u_n)\in\mathbb{R}_n:0<u_1\leq\ldots\leq u_n\}$.
\end{itemize}

 Next, we present a lemma illustrating the preservation of the weak submajorization order in the spaces $\mathcal{D}_n^+$ and $\mathcal{E}_n^+$ (refer to Theorem 1(i) in \cite{haidari2019comparisons}).
\begin{lemma}\label{lemma2.1}
Let $\varphi:\mathcal{D}_n^+ (\mathcal{E}_n^+)\to \mathbb{R}$ be a continuous function on $\mathcal{D}_n^+ (\mathcal{E}_n^+)$
and continuously
differentiable on the interior of $\mathcal{D}_n^+ (\mathcal{E}_n^+)$.
Then, 
\begin{itemize}	
	\item[i)] 
	$\boldsymbol{u}\preccurlyeq_{w}\boldsymbol{v}$ implies 
	$\varphi_{(k)}(\boldsymbol{u})\geq\varphi_{(k)}(\boldsymbol{v})$
	if and only if,  $\varphi_{(k)}$
	is a non-negative
	increasing (non-negative decreasing) function in $k\in\{1, \ldots, n\}$ on  $\mathcal{D}_n^+ (\mathcal{E}_n^+)$.
	
		\item[ii)] $\boldsymbol{u}\stackrel{w}{\preccurlyeq}\boldsymbol{v}$ implies 
	$\varphi_{(k)}(\boldsymbol{u})\leq\varphi_{(k)}(\boldsymbol{v})$
	if and only if,  $\varphi_{(k)}$
	is a non-positive
	decreasing (non-positive increasing) function in $k\in\{1, \ldots, n\}$ on  $\mathcal{D}_n^+ (\mathcal{E}_n^+)$.
	
\end{itemize}
\end{lemma}

The following two lemmas due to \cite{saunders1978quantiles} will be useful to establish the star and dispersive order between the two MRVs.

\begin{lemma}\label{lemma3.1}
	Let $\{F_{\lambda}\mid\lambda\in\mathbb{R}^+\}$ be a class of distribution functions, such that $F_{\lambda}$ is supported on some interval $(a,b)\subseteq(0,\infty)$ and has density $f_{\lambda}$ which does not vanish on any subinterval of $(a,b)$. Then,
	\begin{itemize}
		\item[i)] $F_{\lambda}\leq_{*} F_{\lambda^*}$ for $\lambda\leq\lambda^*$
		if, and only if, $\frac{F^{\prime}_{\lambda}(x)}{xf_{\lambda}(x)}$
		is decreasing in $x$,
		\item[ii)] $F_{\lambda}\leq_{disp} F_{\lambda^*}$ for $\lambda\leq\lambda^*$
		if, and only if, $\frac{F^{\prime}_{\lambda}(x)}{f_{\lambda}(x)}$
		is decreasing in $x$,
	\end{itemize}
	where $F^{\prime}_{\lambda}$ is the derivative of $F_{\lambda}$ with respect to $\lambda$. 
\end{lemma}

To conclude this section, we refer back to a lemma from \cite{kochar2010right}, which will be employed later.
\begin{lemma}\label{lemma4.3}
	Let $\{F_{\lambda}\mid\lambda\in\mathbb{R}^+\}$ be a class of distribution functions, such that $F_{\lambda}$ is supported on some interval $(a,b)\subseteq(0,\infty)$. Then, 
	\begin{eqnarray*}
		F_{\lambda}\leq_{RS}F_{\lambda^*},~~~\lambda\leq\lambda^*
	\end{eqnarray*}	
	if and only if
	\begin{eqnarray*}
		\frac{W^{\prime}_\lambda(x)}{\bar{F}_\lambda(x)}~\text{is increasing in}~x, 
	\end{eqnarray*}
	where $W_\lambda=\int_{x}^{\infty}\bar{F}_\lambda(u)du$ and $W^{\prime}_\lambda$ is the derivative of $W_\lambda$ with respect to $\lambda$.   
\end{lemma}

\section{Basic properties of location-scale finite mixture models}\label{sec_new}

In this section, we  investigate the fundamental properties of FMMs following multiple-outlier $\mathcal{LS}$ families of distributions.
Exploring key characteristics and principles, we lay the foundation for a comprehensive understanding of these models. The discussion encompasses essential concepts that form the basis for subsequent analyses and comparisons within the context of FMMs with $\mathcal{LS}$ distributed components.

Firstly, let us revisit  the definition of finite arithmetic mixture models (FMMs).
Consider a random vector $\boldsymbol{X}=(X_1,\ldots,X_n)$ with $n$ components, where the $i$th component is drawn from the $i$th subpopulation. Suppose there are $n$ homogeneous infinite subpopulations of units, with $X_i$ representing the lifetime of a unit in the $i$th subpopulation, where $i=1,\ldots,n$. The mixture of units drawn from these $n$ subpopulations is represented by the random variable $U_n$.
Let $\bar{F}_i$, $F_i$, and $f_i$ denote the survival function (sf), cumulative distribution function (cdf), and probability density function (pdf) of the $i$th random variable $X_i$, respectively. Then, the sf, cdf, and pdf of the mixture random variable (MRV) $U_n$ are given by:
\begin{eqnarray}\label{eq_mix}
	\bar{F}_{U_n}(x)=\sum\limits_{i=1}^{n}r_i\bar{F}_i(x),
	~F_{U_n}(x)=\sum\limits_{i=1}^{n}r_iF_i(x)
	~\mbox{and}~f_{U_n}(x)=\sum\limits_{i=1}^{n}r_if_i(x),
\end{eqnarray}
where $\boldsymbol{r}=(r_1,\ldots,r_n)$ represents the mixing proportions (weights) with $\sum_{i=1}^{n}r_i=1$ and $r_i\geq0$ for $i\in{1,\ldots,n}$.
Furthermore, we need to revisit the definition of a multiple-outlier model with random variables belonging to the location-scale family.
Consider a random sample $X_1, \ldots, X_{n_1}$ of size $n_1$ from a continuous nonnegative random variable $X^{(1)}\thicksim\mathcal{LS}(F, \sigma_1, \lambda_1)$, and another independent random sample $X_{n_1+1}, \ldots, X_n$ of size $n_2$ from a continuous nonnegative random variable $X^{(2)}\thicksim\mathcal{LS}(F, \sigma_2, \lambda_2)$, where $n_1+n_2=n$.
In other words, the sf, cdf, and pdf of $X^{(i)}$ are given by:
\begin{eqnarray}\label{eq_mo}
	\bar{F}_{i}(x)=\bar{F}\left(\frac{x-\sigma_i}{\lambda_i}\right),
	~F_{i}(x)=F\left(\frac{x-\sigma_i}{\lambda_i}\right)
	~\mbox{and}~f_{i}(x)=\frac{1}{\lambda_i}f\left(\frac{x-\sigma_i}{\lambda_i}\right),
\end{eqnarray}
for $x>\sigma_i$ and $i=1,2$.
Denote by $U$ the MRV representing the finite mixture of $X_1, \ldots, X_n$ with mixing proportions $r_i=r_1$ for $i=1, \ldots, n_1$, and $r_i=r_2$ for $i=n_1+1, \ldots, n$, such that $n_1r_1+n_2r_2=1$.
Combining equations \eqref{eq_mix} and \eqref{eq_mo} allows us to derive the sf, cdf, and pdf characterizing the MRV, but caution is necessary regarding the support of said random variable. For instance, following the definition in \cite{barmalzan2022orderings}, the sf of $U$ would be 
	\[\bar F(x)=n_1 r_1 \bar F \left(\frac{x-\sigma_1}{\lambda_1}\right)+n_2 r_2 \bar F \left(\frac{x-\sigma_2}{\lambda_2}\right),
\quad x\geq\max\{\sigma_1,\sigma_2\},
\]
with $r_1+r_2=1$.
However, this function 
	is not a proper sf. Note that, for $\sigma_1>\sigma_2$, its corresponding pdf is given by
	\[f(x)=\frac{r_1}{\lambda_1}f \left(\frac{x-\sigma_1}{\lambda_1}\right)+
	\frac{r_2}{\lambda_2}f \left(\frac{x-\sigma_2}{\lambda_2}\right),
	\quad x\geq \sigma_1.
	\]
	It is evident that the above function 
	has to integrate to 1.
	However, we get the following
	\begin{eqnarray*}
		\int_{-\infty}^{+\infty}f(x)dx&=&
		\int_{\sigma_1}^{+\infty}f(x)dx=r_1+
		r_2\int_{\sigma_1}^{+\infty}\frac{1}{\lambda_2}f \left(\frac{x-\sigma_2}{\lambda_2}\right)dx\\
		&=&
		1-
		r_2\int_{\sigma_2}^{\sigma_1}\frac{1}{\lambda_2}f \left(\frac{x-\sigma_2}{\lambda_2}\right)dx\neq 1,
	\end{eqnarray*}
	since $\sigma_1>\sigma_2$ and
	\[\int_{\sigma_2}^{+\infty}\frac{1}{\lambda_2}f \left(\frac{x-\sigma_2}{\lambda_2}\right)dx=1.
	\]

Taking into account both the support of the baseline distribution and that of the location-scale model, we obtain, in the following result, a proper  pdf of the MRV $U$.

\begin{proposition}
	Let $(\sigma_1,\sigma_2)\in\mathcal{D}_2^+$ and  $n_1r_1+n_2r_2=1$. The function
\begin{equation}\label{eq_pdf}
f_{U}(x)=\frac{n_1 r_1}{\lambda_1} f\left(\frac{x-\sigma_1}{\lambda_1}\right)I(x>\sigma_1)+
\frac{n_2 r_2}{\lambda_2} f\left(\frac{x-\sigma_2}{\lambda_2}\right)I(x>\sigma_2),
\end{equation}
is a proper pdf  whenever $f$ is a baseline pdf of a nonnegative random variable with unbounded support and  
where $I(x>\sigma_i)$ is an indicator function  which has the form as:
\begin{equation*}
I(x>\sigma_i)=\left\{ 
\begin{array}{ll}
1&\textup{if } x>\sigma_i \\
0&\textup{if } x\leq\sigma_i
\end{array}
\right. 
\end{equation*}
for $i=1,2$.
\end{proposition}

\begin{proof}
	It is evident that $f_{U}$ is non-negative and 
\begin{eqnarray*}
	\int_{-\infty}^{+\infty}f_{U}(x)dx&=&\int_{\sigma_2}^{\sigma_1}n_2r_2f_2(x)dx+
	\int_{\sigma_1}^{+\infty}\big(n_1 r_1 f_1(x)+n_2 r_2 f_2(x)\big)dx\\
	&=&n_1r_1\int_{\sigma_1}^{+\infty} f_1(x)dx+
	n_2r_2\int_{\sigma_2}^{+\infty} f_2(x)dx=n_1r_1+n_2r_2=1,
\end{eqnarray*}
where
\[f_{i}(x)=\frac{1}{\lambda_i}f\left(\frac{x-\sigma_i}{\lambda_i}\right)I(x>\sigma_i), 
\]
such as $f_{i}$ integrates one,
for $i=1,2$. Thus $f_{U}$
is a proper pdf.
\end{proof}

The following counterexample shows that,
if the baseline random variable has bounded support,
then the function defined in \eqref{eq_pdf} 
may not be a proper pdf. 

\begin{counterexample}\label{countex_01}
Let us	consider the baseline distribution as a power distribution with pdf $f(t;a,b)=\frac{a}{b^{a}}t^{a-1}$, for $0<t<b$ with $a,b>0$. Take $a=3$ and $b=2$. 
Let $n_1=3$, $n_2=2$, $r_1=0.1$, $r_2=0.35$, $\sigma_1=4$, $\sigma_2=2$, $\lambda_1=12$ and $\lambda_2=8$.
From \eqref{eq_pdf}, we get
\begin{eqnarray}\label{cont}\nonumber
	f_{U}(x)&=&\frac{0.3}{12} f\left(\frac{x-4}{12}\right)I(x>4)+
\frac{0.7}{8} f\left(\frac{x-2}{8}\right)I(x>2)\\
&=&\frac{3}{320}\left(\frac{x-4}{12}\right)^{2}I(x>4)+
\frac{21}{640}\left(\frac{x-2}{8}\right)^{2}I(x>2).
\end{eqnarray}
Observe that
\[
\int_{4}^{+\infty}\left(\frac{x-4}{12}\right)^{2}dx=+\infty
\quad\textup{and}\quad
\int_{2}^{+\infty}\left(\frac{x-2}{8}\right)^{2}dx=+\infty.
\]
Therefore, in this particular case, the function $f_{U}(x)$ defined in \eqref{cont} is not a proper pdf.
\end{counterexample}

Finally, from \eqref{eq_pdf}, we derive 
the corresponding cdf for the MRV $U$ which is given by:
\begin{equation}\label{eq_cdfU}
F_{U}(x)=\left\{ 
\begin{array}{ll}
0&\textup{if } x\leq\sigma_2 \\
n_2 r_2  F \left(\frac{x-\sigma_2}{\lambda_2}\right)
&\textup{if } \sigma_2<x\leq\sigma_1\\
n_1 r_1  F \left(\frac{x-\sigma_1}{\lambda_1}\right)+n_2 r_2  F \left(\frac{x-\sigma_2}{\lambda_2}\right)
&\textup{if } x>\sigma_1.
\end{array}
\right. 
\end{equation}
Analogously, we have 
the following sf for the MRV $U$:
\begin{equation}\label{eq_sfU}
\bar F_{U}(x)=\left\{ 
\begin{array}{ll}
1&\textup{if } x\leq\sigma_2 \\
n_1 r_1 +n_2 r_2 \bar F \left(\frac{x-\sigma_2}{\lambda_2}\right)
&\textup{if } \sigma_2<x\leq\sigma_1\\
n_1 r_1 \bar F \left(\frac{x-\sigma_1}{\lambda_1}\right)+n_2 r_2 \bar F \left(\frac{x-\sigma_2}{\lambda_2}\right)
&\textup{if } x>\sigma_1.
\end{array}
\right. 
\end{equation}
	
\section{Stochastic comparisons}\label{sec3}

Our focus here is on stochastically comparing
two FMMs with multiple-outlier location-scale family distributed components. 
To do this, we
 consider two FMMs with $n$ and $n^{*}$ number of random variables respectively. 
For the first FMM, we assume that $n_1$ random variables belong to a specific homogeneous subpopulation, while the remaining $n_2$ random variables come from another homogeneous subpopulation, satisfying $n_1+n_2=n$. Similarly, for the second FMM, we consider that $n_1^{*}$ random variables belong to a particular homogeneous subpopulation, and the remaining $n_2^{*}$ random variables are from another homogeneous subpopulation, with $n_1^{*}+n_2^{*}=n^{*}$.

In this section, we present the main results related to various stochastic comparisons. Specifically, we divide this section into two subsections: one for orders of magnitude and another for transform and variability orders. 
 Before
 presenting our main results, we state the following general set-up.

\begin{setup}\label{setup4.1}
	Let $X_1,\ldots,X_{n_1}$ be a random sample of size $n_1$ from a
	continuous nonnegative random  variable $X^{(1)}\thicksim\mathcal{LS}(F,\sigma_1,\lambda_1)$, and $X_{n_1+1},\ldots,X_n$ be another independent random sample of size $n_2$ from a continuous nonnegative random variable $X^{(2)}\thicksim\mathcal{LS}(F,\sigma_2,\lambda_2)$ where $n_1+n_2=n$. Let $Y_1,\ldots,Y_{n_1^*}$ be a random sample of size $n_1^*$ from a continuous nonnegative random variable $Y^{(1)}\thicksim\mathcal{LS}(F,\mu_1,\theta_1)$, and $Y_{n_1^*+1},\ldots,Y_{n^*}$ be another independent random sample of size $n_2^*$ from a continuous nonnegative random variable $Y^{(2)}\thicksim\mathcal{LS}(F,\mu_2,\theta_2)$ where $n_1^*+n_2^*=n^*$. 
	Assume that the baseline cdf $F$ has unbounded support.
	Denote by $U$ the MRV representing the finite mixture of $X_1,\ldots,X_n$ with mixing proportions $r_i=r_1$ for $i=1,\ldots,n_1$ and $r_i=r_2$ for $i=n_1+1,\ldots,n$ such as $n_1r_1+n_2r_2=1$, and $V$ the MRV representing the finite mixture of $Y_1,\ldots,Y_{n^*}$ with mixing proportions $r_i=r_1$ for $i=1,\ldots,n_1^*$ and $r_i=r_2$ for $i=n_1^*+1,\ldots,n^*$
	such as $n_1^*r_1+n_2^*r_2=1$.   
\end{setup}

In what follows, we will use the notation $U_{\lambda}$ when we want to emphasize that the stochastic comparison depends on the lambda parameters, and we will use $U_{n}$ when we want to highlight that the stochastic comparison depends on the number of variables.

\subsection{Magnitude orders}

The focus here lies in understanding the relationships between two MRVs characterized by distinct scale or location parameters. 
This exploration centers on magnitude orders, encompassing criteria such as the usual, hazard rate, reversed hazard rate, and likelihood ratio orders.
In the first part of this subsection, we will investigate the case where the number of random variables composing both MRVs is the same, i.e., $n_i=n_i^*$ for $i=1,2$.
  
\begin{remark}  
Observe that,   under Set-up \ref{setup4.1},
when $n_i=n_i^*$ for $i=1,2$, it is easy to check that $\lambda_i\leq(\geq)~\theta_i$ and $\sigma_i\leq(\geq)~\mu_i$ for all $i=1,2$ imply $X^{(i)}\leq_{st}(\geq_{st})~Y^{(i)}$, for $i=1,2$, and therefore $U\leq_{st}(\geq_{st})~V$. For this reason, we study weaker conditions under which the two MRVs are ordered in some stochastic sense.    
\end{remark}

  In the following result, we assume distinct scale parameter vectors:
  \[\boldsymbol{\lambda}=(\underbrace{\lambda_1,\ldots,\lambda_1}_{n_1\rm\ times},\underbrace{\lambda_2,\ldots,\lambda_2}_{n_2\rm\ times})
  \quad\textup{and}\quad
  \boldsymbol{\theta}=(\underbrace{\theta_1,\ldots,\theta_1}_{n_1\rm\ times},\underbrace{\theta_2,\ldots,\theta_2}_{n_2\rm\ times}),
  \]
while the location parameters are common for both models. Under these conditions, we establish that the weak submajorization order between the vectors of scale parameters implies the usual stochastic ordering between two MRVs $U_{\boldsymbol{\lambda}}$ and $V_{\boldsymbol{\theta}}$.

\begin{theorem}\label{theorem3.1}
    Under Set-up \ref{setup4.1}, let $n_i=n_i^*$ and $\sigma_i=\mu_i$ for $i=1,2$. Suppose the baseline function $F$ is IPRFR
  and
    $(\lambda_1 -\lambda_2)(\sigma_1 - \sigma_2)\geq 0$. 
    Then, for $n_1 r_1 \leq (\geq)~n_2 r_2$,
 \begin{itemize}
 	\item[i)] 
 	$	(\underbrace{\lambda_1,\ldots,\lambda_1}_{n_1\rm\ times},\underbrace{\lambda_2,\ldots,\lambda_2}_{n_2\rm\ times})\preccurlyeq_{w}(\underbrace{\theta_1,\ldots,\theta_1}_{n_1\rm\ times},\underbrace{\theta_2,\ldots,\theta_2}_{n_2\rm\ times})$ implies
 	$U_{\boldsymbol{\lambda}}\geq_{st} V_{\boldsymbol{\theta}}$,    
	\item[ii)] 
	$	(\underbrace{\lambda_1,\ldots,\lambda_1}_{n_1\rm\ times},\underbrace{\lambda_2,\ldots,\lambda_2}_{n_2\rm\ times})\stackrel{w}{\preccurlyeq}(\underbrace{\theta_1,\ldots,\theta_1}_{n_1\rm\ times},\underbrace{\theta_2,\ldots,\theta_2}_{n_2\rm\ times})$ implies
	$U_{\boldsymbol{\lambda}}\geq_{st} V_{\boldsymbol{\theta}}$. 
 \end{itemize}  
\end{theorem}

\begin{proof}
	  Note that $(\lambda_1 -\lambda_2)(\sigma_1 - \sigma_2)\geq 0$ means that
	either $(\lambda_1,\lambda_2)\in\mathcal{D}_2^+$ and $(\sigma_1,\sigma_2)\in\mathcal{D}_2^+$, or 
	$(\lambda_1,\lambda_2)\in\mathcal{E}_2^+$ and $(\sigma_1,\sigma_2)\in\mathcal{E}_2^+$.

  	{\bf (i)}  Under the current conditions, the sf of  $U_{\boldsymbol{\lambda}}$ 
  	is given in \eqref{eq_sfU}, and the sf of 
  	$V_{\boldsymbol{\theta}}$ can be obtained also from \eqref{eq_sfU} by substituting $\lambda_i$ for $\theta_i$ with $i=1,2$. Initially, we consider
  	the case when $(\lambda_1,\lambda_2)\in\mathcal{D}_2^+$ and $(\sigma_1,\sigma_2)\in\mathcal{D}_2^+$. Thus,
  	utilizing Lemma \ref{lemma2.1}(i), the proof of this part is completed by demonstrating that $\partial\bar{F}_{U_{\boldsymbol{\lambda}}}(x)/\partial\lambda_k$ is a non-negative increasing function for $k\in{1,2}$ for all $\boldsymbol{\lambda}$ within the interior of $\mathcal{D}_2^+$.
 For $x>\sigma_1$, 
 the partial derivative of $\bar{F}_{U_{\boldsymbol{\lambda}}}(x)$ with respect to $\lambda_i$ is obtained as
	\begin{eqnarray}\label{eq_20}
	\frac{\partial\bar{F}_{U_{\boldsymbol{\lambda}}}(x)}{\partial\lambda_i}=n_i r_i\frac{x-\sigma_i}{\lambda_i^2}f\left(\frac{x-\sigma_i}{\lambda_i}\right)\geq 0, 
	\end{eqnarray}
for $i=1,2$. 
On the other hand, 
under the assumptions made, we have $\lambda_1\geq\lambda_2$, $\sigma_1\geq\sigma_2$ and $n_1 r_1\leq n_2 r_2$, then clearly we obtain
	\begin{eqnarray}\label{eq_21}
\frac{n_1 r_1}{\lambda_1}
\left(\frac{x-\sigma_1}{\lambda_1}\right)
\tilde{h}\left(\frac{x-\sigma_1}{\lambda_1}\right)
F\left(\frac{x-\sigma_1}{\lambda_1}\right)
\leq
\frac{n_2 r_2}{\lambda_2}
\left(\frac{x-\sigma_2}{\lambda_2}\right)
\tilde{h}\left(\frac{x-\sigma_2}{\lambda_2}\right)
F\left(\frac{x-\sigma_2}{\lambda_2}\right),
	\end{eqnarray}
since $t\tilde{h}(t)$ is increasing in $t$. Therefore,
\[
\frac{\partial\bar{F}_{U_{\boldsymbol{\lambda}}}(x)}{\partial\lambda_1}\leq
\frac{\partial\bar{F}_{U_{\boldsymbol{\lambda}}}(x)}{\partial\lambda_2},
\]
with $x>\sigma_1$.
Now, for $\sigma_2<x\leq\sigma_1$, we have
\[\frac{\partial\bar{F}_{U_{\boldsymbol{\lambda}}}(x)}{\partial\lambda_1}=0\leq
n_2 r_2\frac{x-\sigma_2}{\lambda_2^2}f\left(\frac{x-\sigma_2}{\lambda_2}\right)=
\frac{\partial\bar{F}_{U_{\boldsymbol{\lambda}}}(x)}{\partial\lambda_2}.
\]

For the second case  when  $(\lambda_1,\lambda_2)\in\mathcal{E}_2^+$ and $(\sigma_1,\sigma_2)\in\mathcal{E}_2^+$, 
 we have
\[\frac{\partial\bar{F}_{U_{\boldsymbol{\lambda}}}(x)}{\partial\lambda_1}=
n_1 r_1\frac{x-\sigma_1}{\lambda_1^2}f\left(\frac{x-\sigma_1}{\lambda_1}\right)
\geq
0=
\frac{\partial\bar{F}_{U_{\boldsymbol{\lambda}}}(x)}{\partial\lambda_2},
\]
for $\sigma_1<x\leq\sigma_2$.
On the other hand, for $x>\sigma_2$,
the inequality in \eqref{eq_20}
also holds and we get
\[
\frac{n_1 r_1}{\lambda_1}
\left(\frac{x-\sigma_1}{\lambda_1}\right)
\tilde{r}\left(\frac{x-\sigma_1}{\lambda_1}\right)
F\left(\frac{x-\sigma_1}{\lambda_1}\right)
\geq
\frac{n_2 r_2}{\lambda_2}
\left(\frac{x-\sigma_2}{\lambda_2}\right)
\tilde{r}\left(\frac{x-\sigma_2}{\lambda_2}\right)
F\left(\frac{x-\sigma_2}{\lambda_2}\right),
\]
based on the assumptions made.
Thus, $\partial\bar{F}_{U_{\boldsymbol{\lambda}}}(x)/\partial\lambda_k$ is
a non-negative decreasing function in $k\in\{1,2\}$ for all $\boldsymbol{\lambda}$ in the interior of $\mathcal{E}_2^+$.
Therefore, again from Lemma \ref{lemma2.1}(i),
we obtain the required result.

{\bf (ii)} The distribution function of  $U_{\boldsymbol{\lambda}}$ 
is given in \eqref{eq_cdfU}, and the distribution function of 
$V_{\boldsymbol{\theta}}$ can be obtained also from \eqref{eq_cdfU} by changing $\lambda_i$ for $\theta_i$ with $i=1,2$. Firstly, we consider $(\lambda_1,\lambda_2)\in\mathcal{D}_2^+$ and $(\sigma_1,\sigma_2)\in\mathcal{D}_2^+$. Thus,
with the use of Lemma \ref{lemma2.1}(ii),
the proof of the result gets completed if we show that $\partial{F}_{U_{\boldsymbol{\lambda}}}(x)/\partial\lambda_k$ is
a non-positive decreasing function in $k\in\{1,2\}$ for all $\boldsymbol{\lambda}$ in the interior of $\mathcal{D}_2^+$.
For $x>\sigma_1$, 
the partial derivative of ${F}_{U_{\boldsymbol{\lambda}}}(x)$ with respect to $\lambda_i$ is obtained as
\begin{eqnarray*}
\frac{\partial{F}_{U_{\boldsymbol{\lambda}}}(x)}{\partial\lambda_i}=-n_i r_i\frac{x-\sigma_i}{\lambda_i^2}f\left(\frac{x-\sigma_i}{\lambda_i}\right)\leq 0, 
\end{eqnarray*}
for $i=1,2$. 
On the other hand, 
given the stated assumptions of $\lambda_1\geq\lambda_2$, $\sigma_1\geq\sigma_2$, and $n_1 r_1\leq n_2 r_2$, it is evident that the inequality in \eqref{eq_21} also holds, as $t\tilde{r}(t)$ is increasing in $t$. Therefore,
\[
\frac{\partial{F}_{U_{\boldsymbol{\lambda}}}(x)}{\partial\lambda_1}\geq
\frac{\partial{F}_{U_{\boldsymbol{\lambda}}}(x)}{\partial\lambda_2},
\]
with $x>\sigma_1$.
Now, for $\sigma_2<x\leq\sigma_1$, we have
\[\frac{\partial{F}_{U_{\boldsymbol{\lambda}}}(x)}{\partial\lambda_1}=0\geq
n_2 r_2\frac{x-\sigma_2}{\lambda_2^2}f\left(\frac{x-\sigma_2}{\lambda_2}\right)=
\frac{\partial{F}_{U_{\boldsymbol{\lambda}}}(x)}{\partial\lambda_2}.
\]

For the second case when $(\lambda_1,\lambda_2)\in\mathcal{E}_2^+$ and $(\sigma_1,\sigma_2)\in\mathcal{E}_2^+$, the proof is quite analogous, therefore it is omitted.
\end{proof}


Next, we explore the scenario in which the location parameter vectors
\[\boldsymbol{\sigma}=(\underbrace{\sigma_1,\ldots,\sigma_1}_{n_1\rm\ times},\underbrace{\sigma_2,\ldots,\sigma_2}_{n_2\rm\ times})
\quad\textup{and}\quad
\boldsymbol{\mu}=(\underbrace{\mu_1,\ldots,\mu_1}_{n_1\rm\ times},\underbrace{\mu_2,\ldots,\mu_2}_{n_2\rm\ times})
\]
are different but the scale parameter vector $\boldsymbol{\lambda}$ is common for both multiple-outlier $\mathcal{LS}$ mixture models.
The following two counterexamples illustrate that under conditions similar to those considered in Theorem 4.1, MRVs may not be ordered in the sense of the usual stochastic order when the location parameter vectors are different. Although not detailed in this article for brevity, the various cases we have studied lead us to the conclusion that
 the usual stochastic order among MRVs is not satisfied, in general, when the location parameter vectors are different, as the support of these random variables depends on these parameters.
The only option for the usual stochastic order to hold in this case is for $\sigma_i\leq(\geq)~\mu_i$ for $i=1,2$.
In the first counterexample, we consider 
$(\sigma_1,\sigma_2),(\mu_1,\mu_2)\in\mathcal{E}_2^+$ 
with $\mu_1\leq\sigma_1\leq\sigma_2\leq\mu_2$, while in the second counterexample, we assume $(\sigma_1,\sigma_2),(\mu_1,\mu_2)\in\mathcal{D}_2^+$  and 
$\mu_2\leq\sigma_2\leq\sigma_1\leq\mu_1$.

	\begin{counterexample}\label{counterexample4.2}
	Let us consider a Weibull distribution with cdf $F(t;\alpha)=1-e^{-t^{\alpha}}$ for $t\geq 0$ and $\alpha>0$ as the baseline distribution for the mixture models of multiple-outlier location-scale family of distributions. Assume that $\alpha=2$. Then, it is easy to verify that $F$ is DPRHR 
 Further, let $n_1=n_1^*=3$, $n_2=n_2^*=2$, $n_1r_1=0.51$, $n_2r_2=0.49$, $\lambda_1=\theta_1=2$, $\lambda_2=\theta_2=4$, $\sigma_1=6$, $\sigma_2=8$, $\mu_1=4$, and $\mu_2=12$. 
Therefore, the location parameter vectors verify
\[	\boldsymbol{\sigma}=(\underbrace{6,6,6}_{3\rm\ times},\underbrace{8,8}_{2\rm\ times})\preccurlyeq_{w}(\underbrace{4,4,4}_{3\rm\ times},\underbrace{12,12}_{2\rm\ times})=\boldsymbol{\mu}.\] 
Now, from \eqref{eq_cdfU}, we derive the cdfs of the mixed random variables $U_{\boldsymbol{\sigma}}$ and $V_{\boldsymbol{\mu}}$.		
\begin{equation*}
F_{U_{\boldsymbol{\sigma}}}(t) =
\begin{cases}
0 & \text{if } t \leq 6, \\
0.51\left[1 - e^{-\left(\frac{t-6}{2}\right)^2}\right] & \text{if } 6 < t \leq 8, \\
0.51\left[1 - e^{-\left(\frac{t-6}{2}\right)^2}\right] + 0.49\left[1 - e^{-\left(\frac{t-8}{4}\right)^2}\right] & \text{if } t > 8,
\end{cases}
\end{equation*}
and
\begin{equation*}
F_{V_{\boldsymbol{\mu}}}(t) =
\begin{cases}
0 & \text{if } t \leq 4, \\
0.51\left[1 - e^{-\left(\frac{t-4}{2}\right)^2}\right] & \text{if } 4 < t \leq 12, \\
0.51\left[1 - e^{-\left(\frac{t-4}{2}\right)^2}\right] + 0.49\left[1 - e^{-\left(\frac{t-12}{4}\right)^2}\right] & \text{if } t > 12.
\end{cases}
\end{equation*}
The cdfs of $U_{\boldsymbol{\sigma}}$ and $V_{\boldsymbol{\mu}}$ are depicted in Figure \ref{c1}. Upon inspection of that figure, it is evident that the cdfs intersect. Consequently, we observe that $U_{\boldsymbol{\sigma}}\ngeq_{st}(\nleq_{st})~V_{\boldsymbol{\mu}}$. In other words, the usual stochastic order between $U_{\boldsymbol{\sigma}}$ and $V_{\boldsymbol{\mu}}$ does not hold.    
\end{counterexample}

\begin{counterexample}\label{counterexample4.3}
	Let us consider a Fr\'echet distribution as the baseline distribution for the mixture of multiple-outlier location-scale family distributed components with cdf $F(t)=e^{-t^{-\alpha}}$ for $t>0$ and $\alpha>0$. Assume that $\alpha=3$. It can be easily seen that $F$ is DPRFR.	
 Moreover, let $n_1=n_1^*=4$, $n_2=n_2^*=3$, $r_1=0.1$, $r_2=0.2$, $\lambda_1=\theta_1=6$, $\lambda_2=\theta_2=5$, $\sigma_1=9$, $\sigma_2=6$, $\mu_1=15$, and $\mu_2=2$. 
 Again, the location parameter vectors verify
 \[	\boldsymbol{\sigma}=
 (\underbrace{9,9,9,9}_{4\rm\ times},\underbrace{6,6,6}_{3\rm\ times})\preccurlyeq_{w}(\underbrace{15,15,15,15}_{4\rm\ times},\underbrace{2,2,2}_{3\rm\ times})=\boldsymbol{\mu}.
 \]
In this case, from \eqref{eq_cdfU}, we obtain the 
 cdfs of the MRVs $U_{\boldsymbol{\sigma}}$ and $V_{\boldsymbol{\mu}}$  as follows 
\begin{equation*}
	F_{U_{\boldsymbol{\sigma}}}(t) =
	\begin{cases}
		0 &  \text{if } t \leq 6, \\
		0.6\,e^{-(\frac{t-6}{5})^{-3}} &  \text{if } 6 < t \leq 9, \\
		0.4\,e^{-(\frac{t-9}{6})^{-3}} + 0.6\,e^{-(\frac{t-6}{5})^{-3}} &  \text{if } t > 9,
	\end{cases}
\end{equation*}
and
\begin{equation*}
	F_{V_{\boldsymbol{\mu}}}(t) =
	\begin{cases}
		0 &  \text{if } t \leq 2, \\
		0.6\,e^{-(\frac{t-2}{5})^{-3}} &  \text{if } 2 < t \leq 15, \\
		0.4\,e^{-(\frac{t-15}{6})^{-3}} + 0.6\,e^{-(\frac{t-2}{5})^{-3}} &  \text{if } t > 15.
	\end{cases}
\end{equation*}
The cdfs for $U_{\boldsymbol{\sigma}}$ and $V_{\boldsymbol{\mu}}$ are shown in Figure \ref{c2}. From that figure, it is evident that $U_{\boldsymbol{\sigma}}\ngeq_{st}(\nleq_{st})~V_{\boldsymbol{\mu}}$. In other words, the usual stochastic order between $U_{\boldsymbol{\sigma}}$ and $V_{\boldsymbol{\mu}}$ is not satisfied.
\end{counterexample}

\begin{figure}
	\begin{center}
		\subfigure[]{\label{c1}\includegraphics[width=3.2in]{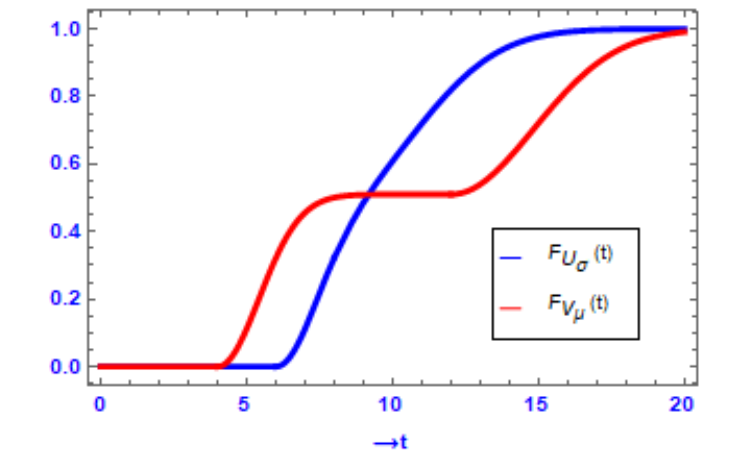}}
		\subfigure[]{\label{c2}\includegraphics[width=3.2in]{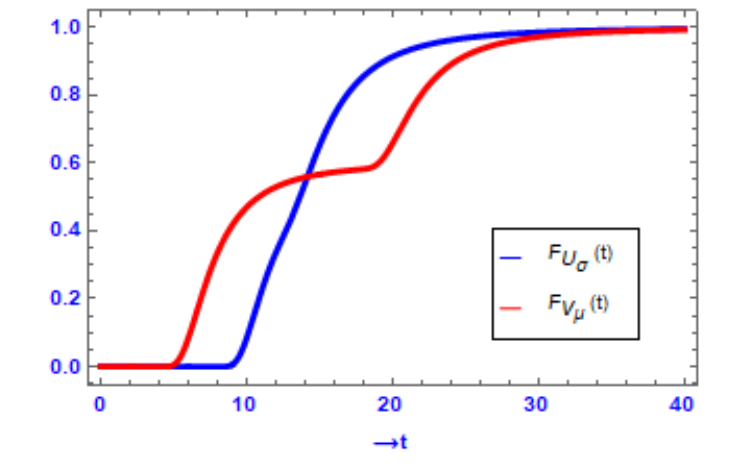}}
	\end{center}
	\caption{(a) Plots of the cdfs of $U_{\boldsymbol{\sigma}}$ and  $V_{\boldsymbol{\mu}}$ in Counterexample \ref{counterexample4.2}. (b) Plots of the cdfs of $U_{\boldsymbol{\sigma}}$ and $V_{\boldsymbol{\mu}}$ in Counterexample \ref{counterexample4.3}.}
\end{figure}

Below, we delve into another interesting problem previously considered in the literature, specifically regarding extreme order statistics. The problem at hand involves studying sufficient conditions for magnitude orders to hold when the number of random variables forming the multiple-outlier models is different, that is, $n_i \neq n_{i}^{*}$ for $i=1,2$. In particular, the study of extreme order statistics from independent variables was explored by \cite{balakrishnan2016comparisons}, whereas \cite{navarro2018comparisons} focused on the case of largest order statistics from multiple-outlier models with dependence. In this work, we investigate the scenario of MRVs arising from multiple-outlier models within the location-scale family.
In the subsequent four outcomes, we investigate the four common orders of magnitude. The subsequent result outlines the sufficient conditions for the MRVs to exhibit the usual stochastic order.

\begin{theorem}\label{theorem3.4}
	Under Set-up \ref{setup4.1}, let $X^{(i)}\stackrel{st}{=}Y^{(i)}$ for $i=1,2$.  
	Assume that 
	$\boldsymbol{r}, \boldsymbol{\lambda}, \boldsymbol{\sigma}\in\mathcal{D}_2^+$
	(or $\boldsymbol{r}, \boldsymbol{\lambda}, \boldsymbol{\sigma}\in\mathcal{E}_2^+$).	
Then, for $n\leq (\geq)~n^{*}$,  we have
$U_{n}\leq_{st}(\geq_{st})~U_{n^{*}}$.
\end{theorem}
\begin{proof}
	 Once again, the sf of $U_{n}$ is provided in \eqref{eq_sfU}, and the sf of $U_{n^{*}}$ can be derived similarly from \eqref{eq_sfU} by substituting $n_i$ with $n_i^{*}$ for $i=1,2$. We only present the proof for the scenario where $\boldsymbol{r}, \boldsymbol{\lambda}, \boldsymbol{\sigma}\in\mathcal{D}_2^+$, as the proof for the alternative case is analogous.
	For $x>\sigma_1$, to prove the required result, we have to show that,
		\begin{equation}\label{eq3.16}
	(n_2 - n_2^{*})r_2 \bar F\left(\frac{x-\sigma_2}{\lambda_2}\right)\leq
	(n_1^{*}-n_1)r_1 \bar F\left(\frac{x-\sigma_1}{\lambda_1}\right).
	\end{equation}
	Because $\bar F(x)$ is a decreasing function
	with respect to $x$ we readily observe that  \eqref{eq3.16} holds under the assumptions.
	For $\sigma_2<x\leq\sigma_1$, it is easy to check that the inequality
	\[(n_2 - n_2^{*})r_2 \bar F\left(\frac{x-\sigma_2}{\lambda_2}\right)\leq
	(n_1^{*}-n_1)r_1
	\]
	also holds under the assumptions. Therefore, 
	the desired result is proven.
\end{proof}

In the next two results, we explore the hazard rate and reversed hazard rate orderings between the two mixed random variables, $U_{n}$ and $U_{n^{*}}$, in the case where the number of random variables in the multiple-outlier models differs, i.e., $n_i \neq n_i^*$ for $i=1,2$.


\begin{theorem}\label{theorem3.6}
Under Set-up \ref{setup4.1}, let $X^{(i)}\stackrel{st}{=}Y^{(i)}$ for $i=1,2$. Suppose $F$ is IFR and
$\boldsymbol{\lambda}, \boldsymbol{\sigma}\in\mathcal{D}_2^+$
or $F$ is DPFR and $ \boldsymbol{\lambda}, \boldsymbol{\sigma}\in\mathcal{E}_2^+$.
Then, 	$U_n\geq_{hr}U_{n^*}$ when
$n_1n_2^*\geq n_1^*n_2$. 
\end{theorem}

\begin{proof}
Firstly, we present the proof for the case 
$\boldsymbol{\lambda}, \boldsymbol{\sigma}\in\mathcal{D}_2^+$.
From (\ref{eq_pdf}) and (\ref{eq_sfU}), we obtain the following 
	hazard rate function of the MRV $U_{n}$ for the case $(\sigma_1,\sigma_2)\in\mathcal{D}_2^+$:
	\begin{equation}\label{eqn_hr}
	 h_{U_n}(x)=\left\{ 
	\begin{array}{ll}
		0& \textup{if } x\leq\sigma_2,\\
		& \\
	\dfrac{\frac{n_2 r_2}{\lambda_2} f(t_2)}{n_1 r_1 +n_2 r_2 \bar F(t_2)}
	&\textup{if } \sigma_2<x\leq\sigma_1,\\
	& \\
	\dfrac{\frac{n_1 r_1}{\lambda_1} f(t_1)+\frac{n_2 r_2}{\lambda_2} f(t_2)}{n_1 r_1\bar F(t_1) +n_2 r_2 \bar F(t_2)}
	&\textup{if } x>\sigma_1,\\
	\end{array}
	\right. 
	\end{equation}
where $t_i=(x-\sigma_i)/\lambda_i$ for $i=1,2$.	
The hazard rate function of 
$U_{n^{*}}$ can be obtained also from (\ref{eqn_hr}) by changing $n_i$ for $n_i^{*}$ with $i=1,2$.   
To prove the desired result, we have to show that 
\begin{equation*}
h_{U_n}(x)\leq h_{U_{n^*}}(x),
\end{equation*}
for all $x\geq 0$. 
Firstly,
for $\sigma_2<x\leq\sigma_1$, 
we need to prove that
\begin{equation*}
\dfrac{n_2}{n_1 r_1 +n_2 r_2 \bar F(t_2)}\leq
\dfrac{n_{2}^{*}}{n_1^{*} r_1 +n_2^{*} r_2 \bar F(t_2)},
\end{equation*}
which is equivalent to
$n_1^{*}n_2\leq n_1n_2^{*}$.
Secondly, we consider $x>\sigma_1$. In this subinterval,
we have to show that 
\begin{equation}\label{eq3.13}
\frac{\frac{n_1 r_1}{\lambda_1} f(t_1)+\frac{n_2 r_2}{\lambda_{2}}f(t_2)}{n_1 r_1 \bar{F}(t_1)+n_2 r_2\bar{F}(t_2)}\leq
\frac{\frac{n_1^* r_1}{\lambda_1}f(t_1)+\frac{n_2^* r_2}{\lambda_2}f(t_2)}{n_1^* r_1 \bar{F}(t_1)+n_2^* r_2\bar{F}(t_2)},
\end{equation}
which is equivalent to
\begin{equation*}
\left[\frac{n_1 r_1}{\lambda_1} f(t_1)+\frac{n_2 r_2}{\lambda_2}f(t_2)\right]\left[{n_1^* r_1 \bar{F}(t_1)+n_2^* r_2\bar{F}(t_2)}\right]
\leq\left[{\frac{n_1^* r_1}{\lambda_1} f(t_1)+\frac{n_2^* r_2}{\lambda_2}f(t_2)}\right]\left[{n_1 r_1 \bar{F}(t_1)+n_2 r_2\bar{F}(t_2)}\right].\nonumber
\end{equation*}
Simplifying the above inequality, we obtain
\begin{equation*}
\frac{r_1 r_2}{\lambda_1}(n_1 n_2^*-n_1^* n_2)f(t_1)\bar{F}(t_2)\leq\frac{r_1 r_2}{\lambda_2}(n_1 n_2^*-n_1^* n_2)f(t_2)\bar{F}(t_1),
\end{equation*}
which is equivalent to
\begin{equation}\label{eq3.14}
\frac{1}{\lambda_1}h(t_1)\leq\frac{1}{\lambda_2}h(t_2),
\end{equation}
since $(n_1n_2^*-n_1^*n_2)\geq 0$ by the assumptions.
Now, since $\lambda_1\geq~\lambda_2$, $\sigma_1\geq~\sigma_2$ and  $h(x)$ is increasing in $x>0$, we have that the inequality \eqref{eq3.14} holds.
Therefore, we get the inequality given in (\ref{eq3.13}) for the case $(\sigma_1,\sigma_2)\in\mathcal{D}_2^+$.
Secondly, in the case where $ \boldsymbol{\lambda}, \boldsymbol{\sigma}\in\mathcal{E}_2^+$, it is worth noting that the proof follows a similar structure up to \eqref{eq3.14}. Subsequently, we reformulate the function in that equation as follows:
\begin{equation*}
\frac{1}{\lambda_i}h(t_i)=\frac{1}{x-\sigma_i}
\left(\frac{x-\sigma_i}{\lambda_i}\right)h\left(\frac{x-\sigma_i}{\lambda_i}\right),
\end{equation*}
for $i=1,2$. Considering that $\lambda_1\leq~\lambda_2$, $\sigma_1\leq~\sigma_2$, and $xh(x)$ is decreasing for $x>0$, it follows that the inequality \eqref{eq3.14} is satisfied.
Consequently, the inequality stated in (\ref{eq3.13}) is obtained when $ \boldsymbol{\lambda}, \boldsymbol{\sigma}\in\mathcal{E}_2^+$.
Thus, the theorem is proven.
\end{proof}

\begin{remark}
	It is worth mentioning that, by carrying out a similar proof to the one in the previous result, it can be concluded that for
	$F$ IFR and
	$\boldsymbol{\lambda}, \boldsymbol{\sigma}\in\mathcal{D}_2^+$
	or $F$ DPFR and $ \boldsymbol{\lambda}, \boldsymbol{\sigma}\in\mathcal{E}_2^+$, 
 	$U_n\leq_{hr}U_{n^*}$ when
	$n_1n_2^*\leq n_1^*n_2$. 
\end{remark}


\begin{theorem}\label{theorem3.7}
	Under Set-up \ref{setup4.1}, let $X^{(i)}\stackrel{st}{=}Y^{(i)}$ for $i=1,2$.
	Suppose $F$ is DPRFR and
	$\boldsymbol{\lambda}, \boldsymbol{\sigma}\in\mathcal{D}_2^+$
	or $F$ is IRFR and $ \boldsymbol{\lambda}, \boldsymbol{\sigma}\in\mathcal{E}_2^+$.
	Then, 	$U_n\geq_{rh}U_{n^*}$ when
	$n_1n_2^*\geq n_1^*n_2$. 	
\end{theorem}

\begin{proof}
	  As before, 
		we firstly present the proof for the case $\boldsymbol{\lambda}, \boldsymbol{\sigma}\in\mathcal{D}_2^+$.
		From \ref{eq_pdf} and \ref{eq_cdfU}, we obtain the following reversed
	hazard rate function of the MRV $U_{n}$ for the case $(\sigma_1,\sigma_2)\in\mathcal{D}_2^+$:
	\begin{equation}\label{eq_hr}
	\tilde{h}_{U_n}(x)=\left\{ 
	\begin{array}{ll}
	0& \textup{if } x\leq\sigma_2,\\
	& \\
	\dfrac{f(t_2)}{\lambda_{2} F(t_2)}
	&\textup{if } \sigma_2<x\leq\sigma_1,\\
	& \\
	\dfrac{\frac{n_1 r_1}{\lambda_1} f(t_1)+\frac{n_2 r_2}{\lambda_{2}}f(t_2)}{n_1 r_1 F(t_1)+n_2 r_2F(t_2)}
	&\textup{if } x>\sigma_1,\\
	\end{array}
	\right. 
	\end{equation}
	where $t_i=(x-\sigma_i)/\lambda_i$ for $i=1,2$.	
	The reversed hazard rate function of 
	$U_{n^{*}}$ can be obtained also from \eqref{eq_hr} by changing $n_i$ for $n_i^{*}$ with $i=1,2$. 
		To prove the required result, we have to show that 
		\begin{equation*}
		\tilde{h}_{U_n}(x)\geq\tilde{h}_{U_{n^*}}(x),
	\end{equation*}
	for all $x\geq 0$. 
	Firstly,
	for $\sigma_2<x\leq\sigma_1$,   it is evident that 
	$\tilde{h}_{U_n}(x)=\tilde{h}_{U_{n^*}}(x)$.
	Secondly, we consider $x>\sigma_1$. In this subinterval,
 we have to show that 
	\begin{equation}\label{eq3.32}
	\frac{\frac{n_1 r_1}{\lambda_1} f(t_1)+\frac{n_2 r_2}{\lambda_{2}}f(t_2)}{n_1 r_1 F(t_1)+n_2 r_2F(t_2)}\geq
	\frac{\frac{n_1^* r_1}{\lambda_1}f(t_1)+\frac{n_2^* r_2}{\lambda_2}f(t_2)}{n_1^* r_1 F(t_1)+n_2^* r_2F(t_2)},
\end{equation}
which is equivalent to
\begin{equation*}
	\left[\frac{n_1 r_1}{\lambda_1} f(t_1)+\frac{n_2 r_2}{\lambda_2}f(t_2)\right]\left[{n_1^* r_1 F(t_1)+n_2^* r_2F(t_2)}\right]
	\geq\left[{\frac{n_1^* r_1}{\lambda_1} f(t_1)+\frac{n_2^* r_2}{\lambda_2}f(t_2)}\right]\left[{n_1 r_1 F(t_1)+n_2 r_2F(t_2)}\right].\nonumber
\end{equation*}
Simplifying the above inequality, we obtain
\begin{equation*}
\frac{r_1 r_2}{\lambda_1}(n_1 n_2^*-n_1^* n_2)f(t_1)F(t_2)\geq\frac{r_1 r_2}{\lambda_2}(n_1 n_2^*-n_1^* n_2)f(t_2)F(t_1),
\end{equation*}
which is equivalent to
\begin{equation}\label{eq3.33}
\frac{1}{\lambda_1}\tilde{h}(t_1)\geq\frac{1}{\lambda_2}\tilde{h}(t_2),
\end{equation}
since $(n_1n_2^*-n_1^*n_2)\geq0$ by the assumptions.
On the other hand, observe that
\begin{equation*}
\frac{1}{\lambda_i}\tilde{h}(t_i)=\frac{1}{x-\sigma_i}
\left(\frac{x-\sigma_i}{\lambda_i}\right)\tilde{h}\left(\frac{x-\sigma_i}{\lambda_i}\right),
\end{equation*}
for $i=1,2$. Now, since $\lambda_1\geq~\lambda_2$,  $\sigma_1\geq~\sigma_2$  and  $x\tilde{h}(x)$ is decreasing in $x>0$, we have that the inequality      \eqref{eq3.33} holds.
Therefore, we get the inequality given in (\ref{eq3.32}) $\boldsymbol{\lambda}, \boldsymbol{\sigma}\in\mathcal{D}_2^+$.
In the second scenario, where $ \boldsymbol{\lambda}, \boldsymbol{\sigma}\in\mathcal{E}_2^+$, the proof closely parallels the argumentation until reaching equation \eqref{eq3.33}. It is important to note that, given $\lambda_1\leq~\lambda_2$ and $\sigma_1\leq~\sigma_2$, coupled with the fact that $\tilde{h}(x)$ is an increasing function for $x>0$, the inequality \eqref{eq3.33} remains valid in this particular case.
Consequently, we establish the inequality as presented in (\ref{eq3.32}), thereby concluding the proof of the theorem.
\end{proof}

\begin{remark}
	As before, it is worth mentioning that, by carrying out a similar proof to the one in the previous result, it can be concluded that for
	$F$ DPRFR and
	$\boldsymbol{\lambda}, \boldsymbol{\sigma}\in\mathcal{D}_2^+$
	or $F$ IRFR and $ \boldsymbol{\lambda}, \boldsymbol{\sigma}\in\mathcal{E}_2^+$, 
	$U_n\leq_{rh}U_{n^*}$ when
	$n_1n_2^*\leq n_1^*n_2$. 
\end{remark}


We are now exploring the possibility of enhancing the conclusions drawn in Theorems \ref{theorem3.6} and \ref{theorem3.7} from hazard rate and reverse hazard rate orderings to likelihood ratio orderings.

\begin{theorem}\label{theorem3.8}	
		Under Set-up \ref{setup4.1}, let $X^{(i)}\stackrel{st}{=}Y^{(i)}$ for $i=1,2$.
		Suppose $F$ is IPLR and
		$\boldsymbol{\lambda}, \boldsymbol{\sigma}\in\mathcal{D}_2^+$
		or $F$ is DLR and $ \boldsymbol{\lambda}, \boldsymbol{\sigma}\in\mathcal{E}_2^+$.
		Then, 	$U_n\geq_{lr}U_{n^*}$ when
		$n_1n_2^*\geq n_1^*n_2$. 	
\end{theorem}

\begin{proof}
	The pdf of the MRVs $U_{n}$ is given in  \ref{eq_pdf}. 
		The pdf of 
	$U_{n^{*}}$ can be obtained also from \eqref{eq_pdf} by changing $n_i$ for $n_i^{*}$ with $i=1,2$. 
	We begin by presenting the proof for the case where $\boldsymbol{\lambda}, \boldsymbol{\sigma}\in\mathcal{D}_2^+$.	
	In the first place, for $\sigma_2<x\leq\sigma_1$, it is easy to check that 
	\[\frac{f_{U_{n}}(x)}{f_{U_{n^*}}(x)}=\frac{n_2}{n_{2}^*},
	\]
	which is both increasing and decreasing in $x$.
	Secondly, we consider $x>\sigma_1$.	
	To establish the desired result, we need to show that 
	\begin{equation*}
	\frac{f_{U_{n}}(x)}{f_{U_{n^*}}(x)}=\frac{\frac{n_1 r_1}{\lambda_1}f\left(\frac{x-\sigma_1}{\lambda_1}\right)+\frac{n_2 r_2}{\lambda_2}f\left(\frac{x-\sigma_2}{\lambda_2}\right) 
	}{\frac{n_1^* r_1}{\lambda_1} f\left(\frac{x-\sigma_1}{\lambda_1}\right)+\frac{n_2^* r_2}{\lambda_2}f\left(\frac{x-\sigma_2}{\lambda_2}\right)}
	\end{equation*} is increasing in $x$. Now, 
	following the idea of the proof of Theorem 3.7 in \cite{torrado2017stochastic},
	the above function can be rewritten as
	\begin{equation}\label{eq3.38}
	\frac{f_{U_{n}}(x)}{f_{U_{n^*}}(x)}=\frac{1}{n_{2}^{*}}\left(n_2 +\frac{n_1 n_{2}^{*}-n_2 n_{1}^*}{n_{1}^{*}+n_{2}^{*}\phi(x)}\right),
\end{equation}
	where
	\[\phi(x)=\frac{\frac{r_2}{\lambda_2}f\left(\frac{x-\sigma_2}{\lambda_2}\right)}{\frac{r_1}{\lambda_1}f\left(\frac{x-\sigma_1}{\lambda_1}\right)}.
	\]
	Differentiating \eqref{eq3.38} with respect to $x$, then we have
\begin{equation}
	\left[\frac{f_{U_{n}}(x)}{f_{U_{n^*}}(x)}\right]^{\prime}=-
	\frac{(n_1 n_{2}^{*}-n_2 n_{1}^*)\phi'(x)}{(n_{1}^*+n_{2}^*\phi(x))^{2}}.
\end{equation}	
Therefore, 	\eqref{eq3.38} is increasing in $x$, if $\phi(x)$ is a decreasing function with respect to $x$, since $n_1 n_{2}^{*}-n_2 n_{1}^*\geq 0$ by assumption.
Now, 		differentiating $\phi(x)$ with respect to $x$, we obtain
\begin{eqnarray}\label{phi}\nonumber
\phi'(x)&\overset{sign}{=}&\frac{r_1r_2}{\lambda_{1}\lambda_{2}^{2}}f'(t_2)f(t_1)-
\frac{r_1r_2}{\lambda_{1}^{2}\lambda_{2}}f'(t_1)f(t_2)\\
&\overset{sign}{=}&
\frac{1}{\lambda_2}\frac{f'(t_2)}{f(t_2)}-
\frac{1}{\lambda_1}\frac{f'(t_1)}{f(t_1)},
\end{eqnarray}
where $t_i=(x-\sigma_i)/\lambda_i$ for $i=1,2$.	
On the other hand, observe that
\begin{equation*}
\frac{1}{\lambda_i}\frac{f'(t_i)}{f(t_i)}
=\frac{1}{x-\sigma_i}
\left(\frac{x-\sigma_i}{\lambda_i}\right)
\frac{f'\left(\frac{x-\sigma_i}{\lambda_i}\right)}{f\left(\frac{x-\sigma_i}{\lambda_i}\right)},
\end{equation*}
for $i=1,2$. Now, since $\lambda_1\geq~\lambda_2$,  $\sigma_1\geq~\sigma_2$  and  $xf'(x)/f(x)$ is decreasing in $x>0$, we have that $\phi'(x)\leq 0$.
For the second case, where $ \boldsymbol{\lambda}, \boldsymbol{\sigma}\in\mathcal{E}_2^+$, the proof follows the same steps up to equation \eqref{phi}. Additionally, considering that $\lambda_1\leq~\lambda_2$, $\sigma_1\leq~\sigma_2$, and $f'(x)/f(x)$ is increasing for $x>0$, we can once again establish that $\phi'(x)\leq 0$.
This concludes the proof of the theorem.	
  \end{proof}

\begin{remark}
	As before, it is worth mentioning that, by carrying out a similar proof to the one in the previous result, it can be concluded that for
	$F$ IPLR and
	$\boldsymbol{\lambda}, \boldsymbol{\sigma}\in\mathcal{D}_2^+$
	or $F$ DLR and $ \boldsymbol{\lambda}, \boldsymbol{\sigma}\in\mathcal{E}_2^+$, 
	$U_n\leq_{lr}U_{n^*}$ when
	$n_1n_2^*\leq n_1^*n_2$. 
\end{remark}

\subsection{Transform and variability orders} 

In this subsection, our attention is directed towards a detailed examination of stochastic orders, namely the star order, the dispersive order, and the right spread order. These orders play a pivotal role in our exploration, focusing specifically on their application in comparing MRVs. From now on, we will use the following notation:
$\lambda_{1:2}=\min\{\lambda_1,\lambda_2\}$, $\lambda_{2:2}=\max\{\lambda_1,\lambda_2\}$, $\theta_{1:2}=\min\{\theta_1,\theta_2\}$, and $\theta_{2:2}=\max\{\theta_1,\theta_2\}$.

In the subsequent result, we focus on a scenario where the number of random variables in the multiple-outlier models is identical, denoted as $n_i=n_i^*$ for $i=1,2$. Additionally, we assume equal location parameters for both mixture observations, set at a constant scalar value $\sigma$.

	\begin{theorem}\label{theorem4.6}
		Under Set-up \ref{setup4.1}, let $n_i=n_i^*$ and $\sigma_1=\sigma_2=\mu_1=\mu_2=\sigma$. Assume that $n_1r_1\geq n_2r_2$. Then, for $\boldsymbol{\lambda},\boldsymbol{\theta}\in\mathcal{D}_2^+$ and fixed $\sigma>0$, we have 
		\begin{eqnarray*}
			\frac{\lambda_{1:2}}{\lambda_{2:2}}\leq\frac{\theta_{1:2}}{\theta_{2:2}}\Rightarrow U_{\boldsymbol{\lambda}}\geq_{*}V_{\boldsymbol{\theta}},
		\end{eqnarray*}
	provided that $F$ is concave and IPLR.             	
	\end{theorem}	

	\begin{proof}
		The sf of $U_{\boldsymbol{\lambda}}$ is given in (\ref{eq_sfU}), and the sf of $V_{\boldsymbol{\theta}}$ can be derived also from (\ref{eq_sfU}) by replacing $\lambda_i$ with $\theta_i$ for $i=1,2$. Let $x>\sigma$. Within this subinterval, the sfs of both $U_{\boldsymbol{\lambda}}$ and $V_{\boldsymbol{\theta}}$ can be written as    
		\begin{eqnarray}\label{eq3.43}
		\bar{F}_{U_{\boldsymbol{\lambda}}}(x)=1-\left[n_1 r_1 F\left((x-\sigma)\lambda_1^*\right)+n_2 r_2 F\left((x-\sigma)\lambda_2^*\right)\right]   
		\end{eqnarray}
		and
		\begin{eqnarray}\label{eq3.44}
		\bar{F}_{V_{\boldsymbol{\theta}}}(x)=1-\left[n_1 r_1 F\left((x-\sigma)\theta_1^*\right)+n_2 r_2 F\left((x-\sigma)\theta_2^*\right)\right],   
		\end{eqnarray}
		where $\lambda_1^*=1/\lambda_1$, $\lambda_2^*=1/\lambda_2$, $\theta_1^*=1/\theta_1$, and $\theta_2^*=1/\theta_2$, respectively.
		To obtain the required result, we have to consider two cases.\\
		{\bf Case-I }: $\lambda_1^*+\lambda_2^*=\theta_1^*+\theta_2^*$.\\
	For convenience, let us assume that $\lambda_1^* + \lambda_2^* = \theta_1^* + \theta_2^* = 1$. In this scenario, it is evident that $(\lambda_1^*,\lambda_2^*)\stackrel{m}{\succcurlyeq}(\theta_1^*,\theta_2^*)$. Now, consider $\lambda^* = \lambda_2^* \geq \lambda_1^*$ and $\theta^* = \theta_2^* \geq \theta_1^*$. Consequently, $\lambda_1^* = 1 - \lambda^*$ and $\theta_1^* = 1 - \theta^*$. With this arrangement, the distribution functions of $U_{\boldsymbol{\lambda}}$ and $V_{\boldsymbol{\theta}}$ can be expressed in the following form:
		\begin{eqnarray}\label{eq3.45}
		F_{U_{\boldsymbol{\lambda}}}(x)\overset{def}{=}F_{\lambda^*}(x)=n_1 r_1 F\left((x-\sigma)(1-\lambda^*)\right)+n_2 r_2 F\left((x-\sigma)\lambda^*\right)
		\end{eqnarray}
		and
		\begin{eqnarray}\label{eq3.46}
		F_{V_{\boldsymbol{\theta}}}(x)\overset{def}{=}F_{\theta^*}(x)=n_1 r_1 F\left((x-\sigma)(1-\theta^*)\right)+n_2 r_2 F\left((x-\sigma)\theta^*\right).
		\end{eqnarray}
		Now, according to Lemma \ref{lemma3.1}(i), we have to show that $\frac{F^{\prime}_{\lambda^*}(x)}{xf_{\lambda^*}(x)}$ is decreasing in $x\in \mathbb{R}^+$, for $\lambda^*\in(1/2,1]$. The derivative of $F_{\lambda^*}(x)$ with respect to $\lambda^*$ is given by
		\begin{eqnarray}\label{eq3.47}
		F^{\prime}_{\lambda^*}(x)=
		(x-\sigma)\left[-n_1 r_1 f\left((x-\sigma)(1-\lambda^*)\right)+n_2 r_2 f\left((x-\sigma)\lambda^*\right)\right].
		\end{eqnarray}
		Also, the pdf corresponding to $F_{\lambda^*}(x)$ is
		\begin{eqnarray}\label{eq3.48}
		f_{\lambda^*}(x)=n_1 r_1 (1-\lambda^*)f\left((x-\sigma)(1-\lambda^*)\right)+n_2 r_2\lambda^* f\left((x-\sigma)\lambda^*\right).
		\end{eqnarray}
		Therefore,
		\begin{eqnarray}\label{eq3.49}
		\frac{F^{\prime}_{\lambda^*}(x)}{xf_{\lambda^*}(x)}&=&\frac{1}{x}\frac{(x-\sigma)\left[-n_1 r_1 f\left((x-\sigma)(1-\lambda^*)\right)+n_2 r_2 f\left((x-\sigma)\lambda^*\right)\right]}{n_1 r_1 (1-\lambda^*)f\left((x-\sigma)(1-\lambda^*)\right)+n_2 r_2\lambda^* f\left((x-\sigma)\lambda^*\right)}\nonumber\\
		&=&\frac{x-\sigma}{x}\left[\frac{n_1 r_1 (1-\lambda^*)f\left((x-\sigma)(1-\lambda^*)\right)+n_2 r_2\lambda^* f\left((x-\sigma)\lambda^*\right)}{-n_1 r_1 f\left((x-\sigma)(1-\lambda^*)\right)+n_2 r_2 f\left((x-\sigma)\lambda^*\right)}\right]^{-1}\nonumber\\
		&=&\frac{x-\sigma}{x}\left[\frac{\lambda^*\left\lbrace n_2 r_2 f\left((x-\sigma)\lambda^*\right)-n_1 r_1 f\left((x-\sigma)(1-\lambda^*)\right)   \right\rbrace+n_1 r_1 f\left((x-\sigma)(1-\lambda^*)\right)  }{n_2 r_2 f\left((x-\sigma)\lambda^*\right)-n_1 r_1 f\left((x-\sigma)(1-\lambda^*)\right)  }\right]^{-1}\nonumber\\
		&=&\frac{x-\sigma}{x}\left[\lambda^*+\frac{n_1 r_1f((x-\sigma)(1-\lambda^*))}{n_2 r_2f((x-\sigma)\lambda^*)-n_1 r_1f((x-\sigma)(1-\lambda^*))}
		\right]^{-1}\nonumber\\
		&=&\frac{x-\sigma}{x}\left[\lambda^*+\left\lbrace \frac{n_2 r_2 f((x-\sigma)\lambda^*)}{n_1 r_1f((x-\sigma)(1-\lambda^*))}-1\right\rbrace ^{-1}\right]^{-1}.
		\end{eqnarray}
		Here $(x-\sigma)/x$ is increasing in $x>0$ and $F^{\prime}_{\lambda^*}$ is negative valued as $f$ is decreasing, $n_2\leq n_1$ and $r_2\leq r_1$. 
		Thus, it suffices to show that the function
		\[L(x)=\frac{n_2 r_2f((x-\sigma)\lambda^*)}{n_1r_1f((x-\sigma)(1-\lambda^*))}\]
		is decreasing in $x\in\mathbb{R}^+$, for $\lambda^*\in(1/2,1]$.
		The derivative of $L(x)$ with respect to $x$ is obtained as
		\begin{eqnarray}\label{eq3.50}
		L^{\prime}&\stackrel{sign}{=}&\frac{\lambda^* f\left((x-\sigma)(1-\lambda^*)\right) f^{\prime}\left((x-\sigma)\lambda^*\right)-(1-\lambda^*) f\left((x-\sigma)\lambda^*\right)f^{\prime}\left((x-\sigma)(1-\lambda^*)\right)}{\left(f((x-\sigma)(1-\lambda^*))\right)^2}\nonumber\\
		&\overset{sign}{=}&\frac{\lambda^* f^{\prime}((x-\sigma)\lambda^*)}{f((x-\sigma)\lambda^*)}-\frac{(1-\lambda^*) f^{\prime}((x-\sigma)(1-\lambda^*))}{f((x-\sigma)(1-\lambda^*))}.
		\end{eqnarray}
		Under the assumptions made, $\frac{x f^{\prime}(x)}{f(x)}$ is decreasing in $x>0$. Therefore, for $\lambda^*\in(1/2,1]$, we have
		\begin{eqnarray}\label{eq3.51}
		\frac{\lambda^* f^{\prime}((x-\sigma)\lambda^*)}{f((x-\sigma)\lambda^*)}\leq\frac{(1-\lambda^*) f^{\prime}((x-\sigma)(1-\lambda^*))}{f((x-\sigma)(1-\lambda^*))}.
		\end{eqnarray}
	Now, by substituting (\ref{eq3.51}) into (\ref{eq3.50}), we obtain $L^{\prime}(x) \leq 0$ for $x \in \mathbb{R}^+$, indicating that $L(x)$ is decreasing for $x > 0$. Consequently, $\frac{F^{\prime}_{\lambda^*}(x)}{x f_{\lambda^*}(x)}$ is a decreasing function for $x > 0$. Therefore, we deduce $U_{\boldsymbol{\lambda}} \geq_{*} V_{\boldsymbol{\theta}}$ based on Lemma \ref{lemma3.1}(i).\\
		{\bf Case-II}. $\lambda_1^*+\lambda_2^*\neq\theta_1^*+\theta_2^*$.\\ 
		In this case, we can consider $\lambda_1^* + \lambda_2^* = k(\theta_1^* + \theta_2^*)$, where $k > 0$ is a scalar. Consequently, $(k\theta_1^*, k\theta_2^*)\stackrel{m}{\preccurlyeq}(\lambda_1^*, \lambda_2^*)$. Let us now examine $n$ independent nonnegative random variables, such that $Z_i \thicksim F(k(x-\sigma)\theta_1^*)$ for $i=1,\ldots,n_1$ and $Z_j \thicksim F(k(x-\sigma)\theta_2^*)$ for $j=n_1+1,\ldots,n$, where $n_1 + n_2 = n$. Then,		
		from {\bf Case-I}, it follows that $U_{\boldsymbol{\lambda}} \geq_{*} Z_{\boldsymbol{\theta}}$. Moreover, the star order is scale-invariant, leading to $U_{\boldsymbol{\lambda}} \geq_{*} V_{\boldsymbol{\theta}}$. This concludes the proof of this result.
	\end{proof}

   The following corollary is a direct consequence of Theorem \ref{theorem4.6} due to the relationship between the star order and the Lorenz order.

	\begin{corollary}
		Under Set-up \ref{setup4.1}, let $n_i=n_i^*$ and $\sigma_1=\sigma_2=\mu_1=\mu_2=\sigma$. Assume that $n_1\geq n_2$ and $r_1\geq r_2$. Then, for $\boldsymbol{\lambda},\boldsymbol{\theta}\in\mathcal{D}_2^+$ and fixed $\sigma>0$, we have 
		\begin{eqnarray*}
			\frac{\lambda_{1:2}}{\lambda_{2:2}}\leq\frac{\theta_{1:2}}{\theta_{2:2}}\Rightarrow U_{\boldsymbol{\lambda}}\geq_{Lorenz}V_{\boldsymbol{\theta}},
		\end{eqnarray*}
		provided that $F$ is concave and IPLR.     	
	\end{corollary}


In the subsequent result, we investigate the dispersive order between the two MRVs $U_{\boldsymbol{\lambda}}$ and $V_{\boldsymbol{\theta}}$ under the condition that the number of random variables in the multiple-outlier models is the same, i.e., $n_i=n_i^*$ for $i=1,2$.

\begin{theorem}\label{theorem4.7}
Under Set-up \ref{setup4.1}, let $n_i=n_i^*$ and $\sigma_1=\sigma_2=\mu_1=\mu_2=\sigma$. Assume that $n_1r_1\geq n_2r_2$. Then, for $\boldsymbol{\lambda},\boldsymbol{\theta}\in\mathcal{D}_2^+$ and fixed $\sigma>0$, we have 
\begin{eqnarray*}
	\frac{\lambda_{1:2}}{\lambda_{2:2}}\leq\frac{\theta_{1:2}}{\theta_{2:2}}\Rightarrow U_{\boldsymbol{\lambda}}\geq_{disp}V_{\boldsymbol{\theta}},
\end{eqnarray*}
provided that $F$ is concave and IPLR.     
\end{theorem}

\begin{proof}
The proof is similar to that of Theorem \ref{theorem4.6}. Thus, it is omitted for the sake of brevity.  
\end{proof}


The following result asserts that the right spread ordering holds for two MRVs, $U_{\boldsymbol{\lambda}}$ and $V_{\boldsymbol{\theta}}$ under some sufficient conditions. In this context, the samples are drawn from multiple-outlier independent location-scale models, specifically when the number of random variables is the same, i.e., $n_i=n_i^*$ for $i=1,2$.

\begin{theorem}\label{theorem4.8}
Under Set-up \ref{setup4.1}, let $n_i=n_i^*$ and $\sigma_1=\sigma_2=\mu_1=\mu_2=\sigma$. Assume that $n_1r_1\leq n_2r_2$. Then, for $\boldsymbol{\lambda},\boldsymbol{\theta}\in\mathcal{E}_2^+$ and fixed $\sigma>0$, we have 
\begin{eqnarray*}
	\frac{\lambda_{1:2}}{\lambda_{2:2}}\leq\frac{\theta_{1:2}}{\theta_{2:2}}\Rightarrow U_{\boldsymbol{\lambda}}\leq_{RS}V_{\boldsymbol{\theta}},
\end{eqnarray*}
provided that $f(t)$ is decreasing in $t>0$.	
\end{theorem}

\begin{proof}
Note that the sf of $U_{\boldsymbol{\lambda}}$ is given in (\ref{eq_sfU}). The sf of $V_{\boldsymbol{\theta}}$ can be derived easily from (\ref{eq_sfU}) by replacing $\lambda_i$ with $\theta_i$ for $i=1,2$. Consider $x>\sigma$. Within this subinterval, the sfs of both $U_{\boldsymbol{\lambda}}$ and $V_{\boldsymbol{\theta}}$ can be respectively written as    
\begin{eqnarray}\label{eq-4.29}
\bar{F}_{U_{\boldsymbol{\lambda}}}(x)=1-\left[n_1 r_1 F\left((x-\sigma)\lambda_1^*\right)+n_2 r_2 F\left((x-\sigma)\lambda_2^*\right)\right]   
\end{eqnarray}
and
\begin{eqnarray}\label{eq-4.30}
\bar{F}_{V_{\boldsymbol{\theta}}}(x)=1-\left[n_1 r_1 F\left((x-\sigma)\theta_1^*\right)+n_2 r_2 F\left((x-\sigma)\theta_2^*\right)\right],   
\end{eqnarray}
where $\lambda_1^*=1/\lambda_1$, $\lambda_2^*=1/\lambda_2$, $\theta_1^*=1/\theta_1$, and $\theta_2^*=1/\theta_2$.
Now, to prove the result, we consider two cases, presented below.\\
{\bf Case-I }: $\lambda_1^*+\lambda_2^*=\theta_1^*+\theta_2^*$.\\
For simplicity, we take $\lambda_1^* + \lambda_2^* = \theta_1^* + \theta_2^* = 1$. Thus, clearly $(\lambda_1^*,\lambda_2^*)\stackrel{m}{\succcurlyeq}(\theta_1^*,\theta_2^*)$. Now, consider $\lambda^* = \lambda_2^* \geq \lambda_1^*$ and $\theta^* = \theta_2^* \geq \theta_1^*$. Consequently, $\lambda_1^* = 1 - \lambda^*$ and $\theta_1^* = 1 - \theta^*$. With this arrangement, the cdfs of $U_{\boldsymbol{\lambda}}$ and $V_{\boldsymbol{\theta}}$ can be expressed in the following form:		
\begin{eqnarray}\label{eq-4.31}
F_{U_{\boldsymbol{\lambda}}}(x)\overset{def}{=}F_{\lambda^*}(x)=n_1 r_1 F\left((x-\sigma)(1-\lambda^*)\right)+n_2 r_2 F\left((x-\sigma)\lambda^*\right)
\end{eqnarray}
and
\begin{eqnarray}\label{eq-4.32}
F_{V_{\boldsymbol{\theta}}}(x)\overset{def}{=}F_{\theta^*}(x)=n_1 r_1 F\left((x-\sigma)(1-\theta^*)\right)+n_2 r_2 F\left((x-\sigma)\theta^*\right).
\end{eqnarray}
From (\ref{eq-4.31}), we obtain  
\begin{eqnarray}\label{eq-4.33}
\bar{F}_{\lambda^*}(x)=n_1r_1\bar{F}((x-\sigma)(1-\lambda^*))+n_2r_2\bar{F}((x-\sigma)\lambda^*),
\end{eqnarray}
where $\lambda^*\leq 1-\lambda^*$ and $\lambda^*\in(0,1/2]$. Similarly, from (\ref{eq-4.32}), we get  
\begin{eqnarray}\label{eq-4.34}
\bar{F}_{\theta^*}(x)=n_1r_1\bar{F}((x-\sigma)(1-\theta^*))+n_2r_2\bar{F}((x-\sigma)\theta^*),
\end{eqnarray}
where $\theta^*\leq 1-\theta^*$ and $\theta^*\in(0,1/2]$. Consider $W_{\lambda^*}=\int_{x}^{\infty}\bar{F}_{\lambda^*}(u)du$. The partial derivative of $W_{\lambda^*}$ with respect to $\lambda^*$ is given by  
\begin{eqnarray}\label{eq-4.35}
W^{\prime}_{\lambda^*}=n_1r_1\lim_{c\rightarrow\infty}\int_{x}^{c}(u-\sigma)f((u-\sigma)(1-\lambda^*))du-n_2r_2\lim_{c\rightarrow\infty}\int_{x}^{c}(u-\sigma)f((u-\sigma)\lambda^*)du.
\end{eqnarray}  
Also,
\begin{eqnarray}\label{eq-4.36}
\frac{\partial}{\partial x}\left(W^{\prime}_{\lambda^*}\right)=-n_1r_1(x-\sigma)f((x-\sigma)(1-\lambda^*))+n_2r_2(x-\sigma)f((x-\sigma)\lambda^*). 
\end{eqnarray}
Now, in $\frac{W^{\prime}_{\lambda^*}(x)}{\overline{F}_{\lambda^*}(x)}$, it is obvious that $\bar{F}_{\lambda^*}(x)$ is decreasing in $x>0$. Thus, if we  show that $W^{\prime}_{\lambda^*}(x)$ is increasing in $x>0$, then $\frac{W^{\prime}_{\lambda^*}(x)}{\overline{F}_{\lambda^*}(x)}$ is increasing in $x>0$. Here, we have taken $\lambda^*\in(0,1/2]$. Now, we have $\lambda^*\leq 1-\lambda^*\Rightarrow (x-\sigma)\lambda^*\leq(x-\sigma)(1-\lambda^*)$. 
Under the assumption made, it can be shown that 
\begin{eqnarray}\label{eq-4.37}
(x-\sigma)f((x-\sigma)\lambda^*)\geq(x-\sigma)f((x-\sigma)(1-\lambda^*)).
\end{eqnarray}    
Using $n_2r_2\geq n_1r_1$ in (\ref{eq-4.37}), we obtain  
\begin{eqnarray}\label{eq-4.38}
-n_1r_1(x-\sigma)f((x-\sigma)(1-\lambda^*))+n_2r_2(x-\sigma)f((x-\sigma)\lambda^*)\geq 0,
\end{eqnarray}  
which implies that $\frac{\partial}{\partial x}(W^{\prime}_{\lambda^*}(x))\geq 0$. That is, $W^{\prime}_{\lambda^*}(x)$ is increasing in $x>0$. Thus, $\frac{W^{\prime}_{\lambda^*}(x)}{\overline{F}_{\lambda^*}(x)}$ is increasing with respect to $x>0$. Hence, by Lemma \ref{lemma4.3}, we get the required result.\\ 
{\bf Case-II}. $\lambda_1^*+\lambda_2^*\neq\theta_1^*+\theta_2^*$.\\ 
In this case, we assume that $\lambda_1^* + \lambda_2^* = k(\theta_1^* + \theta_2^*)$, where $k > 0$ is a scalar. Consequently, $(k\theta_1^*, k\theta_2^*)\stackrel{m}{\preccurlyeq}(\lambda_1^*, \lambda_2^*)$. Let us now examine $n$ independent nonnegative random variables, such that $T_i \thicksim F(k(x-\sigma)\theta_1^*)$ for $i=1,\ldots,n_1$ and $T_j \thicksim F(k(x-\sigma)\theta_2^*)$ for $j=n_1+1,\ldots,n$, where $n_1 + n_2 = n$. Then,		
from {\bf Case-I}, it follows that $U_{\boldsymbol{\lambda}} \leq_{RS} T_{\boldsymbol{\theta}}$. Moreover, the right-spread order is scale-invariant, leading to $U_{\boldsymbol{\lambda}} \leq_{RS} V_{\boldsymbol{\theta}}$. This concludes the proof of this result.       
\end{proof}

\section{Concluding remarks}

The study presented in this paper was motivated by the need to delve deeper into the stochastic comparisons among FMMs, particularly focusing on multiple-outlier $\mathcal{LS}$ families of distributions. Building upon prior research on stochastic comparisons for mixture models, we extended our investigation to encompass various univariate orders of magnitude, transform, and variability within the context of these distributions.

The conceptualization of random variables following the $\mathcal{LS}$ family of distributions, defined by the location and scale parameters, allowed us to analyze scenarios with non-negative independent random variables. This framework widened our understanding of stochastic orders, offering insights into their significance across diverse fields, such as reliability analysis, supply chain management, and insurance studies.
Our research findings, outlined in Section 4, contribute nuanced perspectives to the understanding of stochastic comparisons within the context of FMMs adhering to multiple-outlier $\mathcal{LS}$ families of distributions. The established relationships among various orders of magnitude, transforms, and variability provide a richer comprehension of the intricate interplay between subpopulations in heterogeneous data.

In conclusion, this article contributes to the current discourse on stochastic comparisons within the framework of FMMs, offering valuable insights into the intricate nature of heterogeneous populations and their various subpopulations. We believe that the comprehensive analysis performed here paves the way for future research, fostering a deeper understanding of stochastic orders and their implications in modeling real-world data in multidisciplinary domains.


\section*{Funding} 
Raju Bhakta extends sincere thanks for the financial support received from the Indian Institute of Technology Roorkee under a project grant IITR/SRIC/2301/DRD-2236-CSE/23-24/SR-03, sponsored by Centre for Artificial Intelligence \& Robotics (CAIR), DRDO, Bangalore, Government of India. Sangita Das gratefully acknowledges the financial support for this research work under a grant PDF/2022/00471, SERB, Government of India. The author Nuria Torrado states that this manuscript is part of the project TED2021-129813A-I00 and the grant CEX2019-000904-S. She thanks the support of MCIN/AEI/10.13039/501100011033 and the European Union ``NextGenerationEU''/PRTR.

\section*{Disclosure statement}
All the authors state that there is no conflict of interest.

\bibliography{ref}
\end{document}